\documentclass{article}
\usepackage{amsfonts,amsmath, amssymb,amsthm, latexsym,comment, enumerate,mathrsfs}
\usepackage{hyperref}

\usepackage{multirow}

\newcommand{\C}{\mathbb{C}}
\newcommand{\F}{\mathbb{F}}

\newcommand{\Q}{\mathbb{Q}}

\newcommand{\R}{\mathbb{R}}
\newcommand{\Z}{\mathbb{Z}}

\newcommand{\ba}{\mathbf{a}}

\newcommand{\cA}{\mathcal{A}}
\newcommand{\cB}{\mathcal{B}}

\DeclareMathOperator{\Li}{Li}

\DeclareMathOperator{\sqf}{sqf}

\newtheorem{theorem}{Theorem}[section]
\newtheorem{lemma}[theorem]{Lemma}
\newtheorem{corollary}[theorem]{Corollary}
\newtheorem{definition}[theorem]{Definition}

\usepackage{xcolor,color,url,lmodern}

\newcounter{stagcounter}[section]

\hypersetup{
    colorlinks,
    linkcolor={red!50!black},
    citecolor={blue!50!black},
    urlcolor={blue!80!black}
}

\title{Squarefree discriminants of polynomials with prime coefficients}
\author{Valentio Iverson, Gian Cordana Sanjaya, Xiaoheng Wang}
\date{}

\begin{document}

\maketitle

\begin{abstract}
In this paper, we consider the family of monic polynomials with prime coefficients and the family of all polynomials with prime coefficients. We determine the number of $f(x)$ in each of these families having: squarefree discriminant; $\Z[x]/(f(x))$ as the maximal order in $\Q[x]/(f(x))$. 
\end{abstract}

\section{Introduction}

The problem of determining the number of $n$-tuples $\ba$ of integers such that $P(\ba)$ is squarefree, for some fixed polynomial $P\in\Z[x_1,\ldots,x_n]$, having height $H(\ba) < X$, is a classical one. There have been multiple new results in recent years (\cite{BS2,BS3,BS4,BS5,BSW1,BSW2,Oll,SW}) using new techniques to obtain bounds for the tail estimate:
$$\#\{\ba\in\Z^n\colon H(\ba) < X, m^2\mid F(\ba)\mbox{ for some squarefree }m>M\}.$$
In this paper, we show that the same tail estimates can be used to determine the number of $n$-tuples $\ba$ of \emph{primes} such that $P(\ba)$ is squarefree. As a consequence of the tail estimates proved in \cite{BSW1,BSW2}, we prove the following results.

\begin{theorem}\label{thm:polydisc}
    For each $n \geq 2$ and real number $X>0$, let
    \begin{eqnarray*}
        V_n^m(X) &=& \{x^n + a_1x^{n-1} + \cdots a_n\colon a_i < X^i\mbox{ is prime, for all }i=1,\ldots,n\},\\
        V_n(X) &=& \{a_0x^n + a_1x^{n-1} + \cdots a_n\colon a_i < X\mbox{ is prime, for all }i=0,\ldots,n\},\\
        N_n^{m,\sqf}(X) &=& \#\{f\in V_n^m(X)\colon \Delta(f)\mbox{ is squarefree}\},\\
        N_{n,2}^{m,\sqf}(X) &=& \#\{f\in V_n^m(X)\colon \mbox{the odd part of }\Delta(f)\mbox{ is squarefree}\},\\
        N_n^{\sqf}(X) &=& \#\{f\in V_n(X)\colon \Delta(f)\mbox{ is squarefree}\},\\
        N_{n,2}^{\sqf}(X) &=& \#\{f\in V_n(X)\colon \mbox{the odd part of }\Delta(f)\mbox{ is squarefree}\},\\
        N_n^{m,\max}(X) &=& \#\{f\in V_n^m(X)\colon \Z[x]/(f(x))\mbox{ is the maximal order of }\Q[x]/(f(x))\},\\
        N_{n,2}^{m,\max}(X) &=& \#\{f\in V_n^m(X)\colon \Z[x]/(f(x))\mbox{ is the maximal order of }\Q[x]/(f(x))\mbox{ except possibly at }2\},\\
        N_n^{\max}(X) &=& \#\{f\in V_n(X)\colon \Z[x]/(f(x))\mbox{ is the maximal order of }\Q[x]/(f(x))\},\\
        N_{n,2}^{\max}(X) &=& \#\{f\in V_n(X)\colon \Z[x]/(f(x))\mbox{ is the maximal order of }\Q[x]/(f(x))\mbox{ except possibly at }2\}.
    \end{eqnarray*}
    Then for any real number $A>0$, we have
    \begin{eqnarray*}
        N_n^{m,\sqf}(X) &=& C_n^{\sqf} \prod_{i = 1}^n \Li(X^i) + O_{n, A}\left(\frac{X^{n(n + 1)/2}}{(\log X)^A}\right), \\
        N_{n, 2}^{m,\sqf}(X) &=& C_{n, 2}^{\sqf} \prod_{i = 1}^n \Li(X^i) + O_{n, A}\left(\frac{X^{n(n + 1)/2}}{(\log X)^A}\right),\\
        N_n^{\sqf}(X) &=& C_n^{\sqf} \prod_{i = 0}^n \Li(X) + O_{n, A}\left(\frac{X^{n + 1}}{(\log X)^A}\right), \\
        N_{n, 2}^{\sqf}(X) &=& C_{n, 2}^{\sqf} \prod_{i = 0}^n \Li(X) + O_{n, A}\left(\frac{X^{n + 1}}{(\log X)^A}\right),\\
    \end{eqnarray*}
    \begin{eqnarray*}
        N_n^{m,\max}(X) &=& C_n^{\max} \prod_{i = 1}^n \Li(X^i) + O_{n, A}\left(\frac{X^{n(n + 1)/2}}{(\log X)^A}\right), \\
        N_{n, 2}^{m,\max}(X) &=& C_{n, 2}^{\max} \prod_{i = 1}^n \Li(X^i) + O_{n, A}\left(\frac{X^{n(n + 1)/2}}{(\log X)^A}\right),\\
        N_n^{\max}(X) &=& C_n^{\max} \prod_{i = 0}^n \Li(X) + O_{n, A}\left(\frac{X^{n + 1}}{(\log X)^A}\right), \\
        N_{n, 2}^{\max}(X) &=& C_{n, 2}^{\max} \prod_{i = 0}^n \Li(X) + O_{n, A}\left(\frac{X^{n + 1}}{(\log X)^A}\right),\\
    \end{eqnarray*}
    where
    \[ C_n^{\sqf} = \prod_p P_{n,p}^{\sqf}, \qquad
    C_{n, 2}^{\sqf} = \prod_{p > 2} P_{n,p}^{\sqf}, \qquad
    C_n^{\max} = \prod_p P_{n,p}^{\max}, \qquad
    C_{n, 2}^{\max} = \prod_{p > 2} P_{n,p}^{\max}, \]
    for some local densities $P_{n,p}^{\sqf},P_{n,p}^{\max}$  satisfying
    \begin{eqnarray*}
        P_{n,p}^{\sqf} &=& 1 - \frac{3p - 1}{p(p + 1)^2} + O\left(\min\left\{\left(\frac{p}{(p - 1)^2}\right)^{5/2}, \left(\frac{p}{(p - 1)^2}\right)^{\lfloor\sqrt{\log n/\log p}\rfloor/2}\right\}\right), \\
        P_{n,p}^{\max} &=& 1 - \frac{1}{p^2 + p + 1} + O\left(\min\left\{\frac{1}{(p - 1)^4}, \frac{1}{(p - 1)^{2\lfloor\sqrt{\log n/(4\log p)}\rfloor}}\right\}\right),
    \end{eqnarray*} 
    for $p>2$ and $n\geq 16$, with $$P_{n,2}^{\sqf} = \begin{cases}1&\mbox{ if }n\mbox{ is even}\\0&\mbox{ if }n\mbox{ is odd}\end{cases},\qquad P_{n,2}^{\max} = \begin{cases}1&\mbox{ if }n+1\mbox{ is odd}\\2p_t&\mbox{ if }n+1\equiv 2\pmod{4}\\ p_t&\mbox{ if }n+1\equiv 0\pmod{4}\end{cases},$$ where $t$ is the odd part of $n + 1$ and
    $$p_t = \prod_{d \mid t} \left(1 - \frac{1}{2^{o_d(2)}}\right)^{\phi(d)/o_d(2)},$$
    where $o_d(2)$ for any odd positive integer $d$ denotes the order of $2$ in $(\Z/d\Z)^\times$.
    Moreover, $C_{n, 2}^{\sqf}$ and $C_{n, 2}^{\max}$ are positive for every $n\geq 2$ and
    \begin{eqnarray*}
        \lim_{n \to \infty} C_{n, 2}^{\sqf} &=& \prod_{p > 2} \left(1 - \frac{3p - 1}{p(p + 1)^2}\right) \,\,\,\approx\,\,\, 67.69\%,\\
        \lim_{n \to \infty} C_{n, 2}^{\max} &=& \prod_{p > 2} \left(1 - \frac{1}{p^2 + p + 1}\right) \,\,\,\approx\,\,\, 85.26\%.
    \end{eqnarray*}
\end{theorem}

It is worth noting that the corresponding limits of the local densities for squarefree and maximality except possibly at $2$ without the prime conditions obtained by Yamamura \cite[Proposition~3]{Yama} and Lenstra \cite[Proposition~3.5]{ABZ} are 
$$\prod_{p>2}\left(1 - \frac{3p - 1}{p^2(p + 1)}\right) \quad\mbox{and}\quad \prod_{p > 2} \left(1 - \frac{1}{p^2}\right)$$
respectively.

As a consequence of the uniformity estimate proved in \cite{SW}, we have the following result.
\begin{theorem}\label{thm:a4b3}
    The number of pairs $(a, b)$ of primes such that $a < X^3$, $b < X^4$, and $a^4 + b^3$ is squarefree is
\[ C \Li(X^3)\Li(X^4) + O_A\left(\frac{X^7}{(\log X)^A}\right), \]
    for any $A>0$, where \[ C = \prod_p \left(1 - \frac{1}{p^2 - p}\right) \approx 37.40\%. \]
\end{theorem}

Theorem \ref{thm:polydisc} and Theorem \ref{thm:a4b3} follow from:
\begin{theorem}\label{thm:primesieve}
Let  $n$ and $N > 1$ be positive integers. Let $d_1, d_2, \ldots, d_n$ be fixed positive real numbers and let $d = d_1 + \cdots + d_n$.
Define a height function on $\Z^n$ by the formula
\[ H(\mathbf{a}) = \max_{1 \leq i \leq n} |a_i|^{1/d_i}. \]
Let $c > 1$ be a real number. For each prime $p$, let $\cB_p$ be a subset of $\Z_p^n$ with measure $\lambda_p(\cB_p) \ll p^{-c}$ such that $\cB_p$ is defined by congruence conditions mod $p^N$, where $\lambda_p$ denotes the $p$-adic Haar measure normalized so that $\lambda_p(\Z_p^n) = 1$. For any squarefree integer $m$, let $\cB_m = \bigcap_{p\mid m}(\cB_p\cap\Z^n)$ and define
\[ \rho'(m^N) = \#\{\mathbf{v}=(v_1,\ldots,v_n) \bmod{m^N} : \gcd(v_1, m) = \ldots = \gcd(v_n, m) = 1, \mathbf{v} \in \cB_m\}. \] Suppose $\{\cB_p\}_p$ satisfy the following Uniformity Estimate: there exist real numbers $\alpha,\beta>0$ such that for any real number $M>0$,
\begin{equation}\label{eq:unif}
    \sum_{\substack{m > M\\ m\text{ squarefree}}}\#\{\ba\in\Z^n\colon H(\ba) < X, \ba\in\cB_m\} = O_\epsilon\left(\frac{X^{d+\epsilon}}{M^\alpha}\right) + O_\epsilon(X^{d - \beta + \epsilon}).
\end{equation}
Let
\[ N'_{\cB}(X) = \#\{\mathbf{q}=(q_1,\ldots,q_n) \in \Z^n : q_1, \ldots, q_n \text{ prime, } H(\mathbf{q}) < X, \mathbf{q} \notin \cB_p \text{ for all $p$ prime}\}. \]
Then for any $A>0$, we have
$$N'_{\cB}(X) = C'_{\cB} \prod_{i = 1}^n \Li(X^{d_i}) + O_{A}\left(\frac{X^{d}}{(\log X)^A}\right),$$
where
 \[ C'_{\cB} = \prod_p\left(1 - \frac{\rho'(p^N)}{\phi(p^N)^n}\right). \]
Furthermore, $C'_{\cB} > 0$ provided that $\rho'(p^N) < \phi(p^N)^n$ for each $p$ prime; that is, there exists $\ba\in\Z^n\backslash \cB_p$ having all coordinates coprime to $p$.
\end{theorem}

When the bad set $\cB_p$ is contained in a set of the form $\{\ba\in\Z_p^n\colon p^2\mid F(\ba)\}$
for some fixed polynomial $F(x_1,\ldots,x_n)\in\Z[x_1,\ldots,x_n]$, we have $\lambda_p(\cB_p)\ll p^{-2}$ by \cite[Lemma 2.1]{LX1}.

We prove Theorem \ref{thm:primesieve} following the technique of Browning \cite{BR} using the inclusion-exclusion sieve and then splitting the sum into the small range $m < (\log X)^{\eta}$ for some large $\eta$, the medium range $(\log X)^{\eta} < m < X^\delta$ for some small $\delta$, and the large range $m > X^\delta$ where we use our uniformity estimate. The local densities in Theorem \ref{thm:a4b3} are easily computed. 

The bulk of the paper is devoted to the calculation of the local densities in Theorem \ref{thm:polydisc}. Let $V_n(\Z_p)$ denote the set of monic degree $n$ polynomials in $\Z_p[x]$ and let $U_n(\Z_p)$ denote the subset of $V_n(\Z_p)$ where all the coefficients are not divisible by $p$. Then $P^{\sqf}_{n,p}$ and $P^{\max}_{n,p}$ are the squarefree density and maximality density of $f(x)$ in $U_n(\Z_p)$. Since a non-monic polynomial with leading coefficient not divisible by $p$ behaves just like a monic polynomial over $\Z_p$, the corresponding densities in the space of non-monic polynomials are the same as in the space of monic polynomials. When using the inclusion-exclusion sieve, an important quantity that we need to estimate is
$$\#\{h\in U_n(\F_p)\colon u\mid h\}$$
for some fixed monic polynomial $u\in\F_p[x]$, where $U_n(\F_p)$ is the set of monic degree $n$ polynomial in $\F_p[x]$ with all coefficients nonzero. We expect the residual classes of polynomials in $U_n(\F_p)$ modulo $u$ to be equidistributed. We approximate $\#\{h\in U_n(\F_p)\colon u\mid h\}$ by $p^{-\deg(u)}\#U_n(\F_p)$ to obtain the main terms in $P^{\sqf}_{n,p}$ and $P^{\max}_{n,p}$. We note that when $x\mid u$, the set $\{h\in U_n(\F_p)\colon u\mid h\}$ is empty. This is the main cause of the difference in densities in the space of all integer polynomials compared to the densities in the space of polynomials with prime coefficients, in the $n$-limit. We remark that the same method can be used to determine the squarefree and maximality densities in the space of all integer polynomials with \emph{prime constant coefficients} and they will have the same limits as $P^{\sqf}_{n,p}$ and $P^{\max}_{n,p}$ as $n$ goes to infinity (see \cite{ST}).

We prove upper bounds on the discrepancy
$$\left|\frac{\#\{h\in U_n(\F_p)\colon u\mid h\}}{\#U_n(\F_p)} - \frac{1}{p^{\deg(u)}}\right|$$
by modeling a degree $n$ polynomial in $U_n(\F_p)$ mod $u$ as a path of length $n$ on a certain directed graph whose vertices are indexed by $\F_p[x]/(u)$, and where every vertex has in-degree and out-degree $p - 1$. We use the theory of double stochastic matrices to estimate the sizes of the entries of powers of the adjacency matrix.

This paper is organized as follows. In Section 2, we perform the sieve to prime inputs and prove Theorem \ref{thm:primesieve}. In Section 3, we compute the densities $P^{\sqf}_{n,p}$ and $P^{\max}_{n,p}$ in terms of the quantities $\#\{h\in U_n(\F_p)\colon u\mid h\}$ and show how the main terms appear. In Section 4, we bound the discrepancy and complete the proof of Theorem \ref{thm:polydisc}.

\section{Proof of Theorem~\ref{thm:primesieve}}

In this section, we prove Theorem~\ref{thm:primesieve}, following the method of Browning in~\cite{BR}. By the inclusion-exclusion sieve:
\begin{equation}\label{eq:incexc}
    N'_{\cB}(X) = \sum_m \mu(m) N'_{\cB}(m, X) 
\end{equation}
where $$N'_{\cB}(m, X) = \#\{\mathbf{q}=(q_1,\ldots,q_n) \in \Z^n : q_1, \ldots, q_n \text{ prime, } H(\mathbf{q}) < X, \ba\in\cB_m\}.$$
Let $\eta,\delta>0$ be real numbers to be chosen later. We split the sum in \eqref{eq:incexc} into
\begin{equation}
    \label{eq:split}
    N'_{\cB}(X) = \sum_{m \leq (\log X)^{\eta}} \mu(m) N'_{\cB}(m, X) + \sum_{(\log X)^{\eta} < m \leq X^{\delta}} \mu(m) N'_{\cB}(m, X) + \sum_{m > X^{\delta}} \mu(m) N'_{\cB}(m, X). 
\end{equation}
The assumption $\lambda_p(\cB_p) \ll p^{-c}$ then yields $\rho'(p^N) \ll p^{Nn - c}$ for each prime $p$.
For each squarefree $m$, we then have $\rho'(m^N) \ll_{\epsilon} m^{Nn - c + \epsilon}$. Moreover, we have the general estimate
$$N'_{\cB}(m, X)
    = \sum_{\substack{\mathbf{v} \bmod{m^N} \\ \forall i, \gcd(v_i, m) = 1 \\ \mathbf{v} \in \cB_m}}
        \#\{\mathbf{q} \in \Z^n : q_1, \ldots, q_n \text{ prime, } H(\mathbf{q}) < X, \mathbf{q} \equiv \mathbf{v} \bmod{m^N}\} + O_{\epsilon}\left(\sum_{i = 1}^n m^{\epsilon} X^{d - d_i}\right),$$
where the error term comes from when at least one of the $q_i$ is one of the $m^\epsilon$ prime divisors of $m$.

We now use the Siegel-Walfisz theorem to count primes mod $m^N$ where $m < (\log X)^\eta$. Let $d_{\min} = \min_i d_i$. There exists a constant $C_1>0$ depending on $N\eta$ (and $d_{\min}$) such that for any squarefree $m \leq (\log X)^\eta$,
\begin{eqnarray*}
    N'_{\cB}(m, X) &=& \sum_{\substack{\mathbf{v} \bmod{m^N} \\ \forall i, \gcd(v_i, m) = 1 \\ \mathbf{v} \in \cB_m}}\prod_{i=1}^n \left(\frac{\Li(X^{d_i})}{\phi(m^N)} + O\Big(X^{d_i}\exp(-C_1\sqrt{\log X})\Big)\right) + O_{\epsilon}\left(\sum_{i = 1}^n m^{\epsilon} X^{d - d_i}\right)\\
    &=& \frac{\rho'(m^N)}{\phi(m^N)^n}\prod_{i=1}^n\Li(X^{d_i}) + O\left(\rho'(m^N)X^d\exp(-C_1\sqrt{\log X})\right) + O_\epsilon(X^{d - d_{\min}+\epsilon})\\
    &=& \frac{\rho'(m^N)}{\phi(m^N)^n}\prod_{i=1}^n\Li(X^{d_i}) + O\left(X^d\exp(-C_2\sqrt{\log X})\right),
\end{eqnarray*}
where we may take $C_2 = C_1/2$ since $\rho'(m^N) \ll (\log X)^{Nn\eta}$. Summing over $m<(\log X)^\eta$ gives
$$\sum_{m \leq (\log X)^{\eta}} \mu(m) N'_{\cB}(m, X) = \left(\sum_{m \leq (\log X)^{\eta}} \mu(m) \frac{\rho'(m^N)}{\phi(m^N)^n}\right)\prod_{i=1}^n\Li(X^{d_i}) + O\left(X^d\exp(-C_3\sqrt{\log X})\right)$$
where $C_3 = C_2/2$. Since $\rho'$ is multiplicative, we have
\begin{eqnarray*}
    \prod_p\left(1 - \frac{\rho'(p^N)}{\phi(p^N)^n}\right) &=& \sum_{m \leq (\log X)^{\eta}} \mu(m) \frac{\rho'(m^N)}{\phi(m^N)^n} + O\left(\sum_{m > (\log X)^{\eta}}\frac{\rho'(m^N)}{\phi(m^N)^n}\right)\\
    &=& \sum_{m \leq (\log X)^{\eta}} \mu(m) \frac{\rho'(m^N)}{\phi(m^N)^n} + O_\epsilon\left(\sum_{m > (\log X)^{\eta}}\frac{m^{Nn - c + \epsilon}}{m^{Nn-\epsilon}}\right)\\
    &=& \sum_{m \leq (\log X)^{\eta}} \mu(m) \frac{\rho'(m^N)}{\phi(m^N)^n} + O_\epsilon\left(\frac{1}{(\log X)^{\eta(c - 1)-\epsilon}}\right).
\end{eqnarray*}
Hence we have 
\begin{equation}\label{eq:small}
    \sum_{m \leq (\log X)^{\eta}} \mu(m) N'_{\cB}(m, X) = \prod_p\left(1 - \frac{\rho'(p^N)}{\phi(p^N)^n}\right)\prod_{i=1}^n\Li(X^{d_i}) + O\left(\frac{X^d}{(\log X)^{n + \eta(c-1) - \epsilon}}\right).
\end{equation}

For the medium range $(\log X)^\eta \leq m < X^\delta$, we ignore the prime condition. Moreover, we may assume $\delta < d_{\min} / N$ so that
$$N'_{\cB}(m, X) \ll \rho'(m^N) \prod_{i = 1}^n \frac{X^{d_i}}{m^N} \ll_{\epsilon} m^{-c + \epsilon} X^{d}.$$
Summing gives
\begin{equation}\label{eq:mid}
    \sum_{(\log X)^{\eta} \leq m < X^\delta} \mu(m) N'_{\cB}(m, X) \ll_\epsilon \frac{X^d}{(\log X)^{\eta(c - 1)-\epsilon}}
\end{equation}

Finally for the large range $m \geq X^\delta$, we use the uniformity estimate assumption to obtain
\begin{equation}\label{eq:large}
    \sum_{m \geq X^\delta} \mu(m) N'_{\cB}(m, X) \ll_\epsilon X^{d-\delta\alpha+\epsilon}+ X^{d -\beta+\epsilon}
\end{equation}
Therefore, given any $A > 0$, by taking $\eta = (2A+1)/(c - 1)$ and $\delta = d_{\min}/(2N)$, we have
$$N'_{\cB}(X) = \prod_p\left(1 - \frac{\rho'(p^N)}{\phi(p^N)^n}\right)\prod_{i=1}^n\Li(X^{d_i}) + O_A\left(\frac{X^d}{(\log X)^A}\right).$$

If $\rho'(p^N) < \phi(p^N)^n$ for all primes $p$, then the local density $1 - \frac{\rho'(p^N)}{\phi(p^N)^n}$ is positive and is $\gg 1 - p^{-c}$ by the assumption $\lambda_p(\cB_p)\ll p^{-c}$. Hence their product over all $p$ converges to a positive number.
This concludes the proof of Theorem \ref{thm:primesieve}.
\section{Density calculation}

Let $V_n(\Z_p)$ denote the set of monic degree $n$ polynomials in $\Z_p[x]$ and let $U_n(\Z_p)$ denote the subset of $V_n(\Z_p)$ where all the coefficients are not divisible by $p$. In this section, we work towards computing the squarefree density $P^{\sqf}_{n,p}$ and maximality density $P^{\max}_{n,p}$ of $f(x)$ in $U_n(\Z_p)$.

Let $\F_p[x]_m$ denote the set of monic polynomials over $\F_p$. Let $U_n(\F_p)$ be the set of monic degree $n$ polynomials in $\F_p[x]$ with all coefficients nonzero. For a fixed $u\in \F_p[x]_m$, and for every $\alpha\in\F_p[x]/(u)$, we define the set
$$\cA_n(u;\alpha) = \{(a_1, a_2, \ldots, a_n) \in (\F_p \setminus \{0\})^n : a_1 x^{n - 1} + \ldots + a_n \equiv \alpha \bmod{u}\}.$$
We expect the $(p-1)^n$ polynomials $a_1x^{n-1} + \cdots + a_n\in\F_p[x]$ with nonzero coefficients to be somewhat equidistributed in the $p^{\deg(u)}$ residue classes mod $u$. In other words, we expect the following quantity
\begin{equation}\label{eq:delta}
    \delta_{n,p}(d) = \max_{\substack{u \in \F_p[x]_m \\ \deg(u) = d \\ x \nmid u}} \max_{\alpha \in \F_p[x]/(u)} \left|\frac{|\cA_{n}(u; \alpha)|}{(p - 1)^n} - \frac{1}{p^d}\right|
\end{equation} 
to be small (see Theorem \ref{thm3:matrix-p-large-bound} and Theorem \ref{thm3:matrix-p-small-bound}). Notice from the definition that $\delta_{n,p}(0) = 0$.  The goal of this section is to prove the following results expressing $P^{\sqf}_{n,p}$ and $P^{\max}_{n,p}$ in terms of $\delta_{n,p}(d)$.

\begin{theorem}\label{thm:Psqf}
For any positive integer $n \geq 2$,
\[ P_{n, 2}^{\sqf} = \begin{cases} 1, & \text{if } 2 \mid n, \\ 0, & \text{if } 2 \nmid n. \end{cases} \]
For any odd prime $p$ and $n \geq 2$, we have
\[ P_{n,p}^{\sqf} = 1 - \frac{3p - 1}{p(p + 1)^2} + O\left(\sum_{d = 0}^n p^{d/2} \delta_{n,p}(d) + \frac{1}{p^{n/2}}\right), \]
    where the implied constant is absolute.
\end{theorem}

\begin{theorem}\label{thm:Pmax}
For any positive integer $n \geq 2$,
$$P_{n,2}^{\max} = \begin{cases}1&\mbox{ if }n+1\mbox{ is odd}\\2p_t&\mbox{ if }n+1\equiv 2\pmod{4}\\ p_t&\mbox{ if }n+1\equiv 0\pmod{4}\end{cases},$$ where $t$ is the odd part of $n + 1$ and
    $$p_t = \prod_{d \mid t} \left(1 - \frac{1}{2^{o_d(2)}}\right)^{\phi(d)/o_d(2)},$$
    where $o_d(2)$ for any odd positive integer $d$ denotes the order of $2$ in $(\Z/d\Z)^\times$.
For any odd prime $p$ and $n \geq 2$, we have
\[ P_{n,p}^{\max} = 1 - \frac{1}{p^2 + p + 1} + O\left(\sum_{d = 0}^{\lfloor n/2 \rfloor} \delta_{n,p}(2d) + \frac{1}{p^n}\right), \]
    where the implied constant is absolute.
\end{theorem}

\subsection{Proof of Theorem~\ref{thm:Psqf}}

We start by characterizing monic polynomials in $\Z_p$ with squarefree discriminant.

\begin{lemma}\label{lem4:sqfree-discr-description}
(\cite[Proposition 6.7]{ABZ}) Let $p$ be a prime and $f \in \Z_p[x]$ be a monic non-constant polynomial.
Denote by $\Delta(f)$ the discriminant of $f$, and by $\overline{f}$ the reduction of $f$ modulo $p$.
\begin{itemize}
    \item   $\Delta(f)$ is a unit if and only if $\overline{f} \in \F_p[x]$ is squarefree.
    \item   $\Delta(f)$ has valuation $1$ if and only if $p \neq 2$ and the following holds: there exists $c \in \F_p$ and $g \in \F_p[x]_m$ squarefree such that $x + c \nmid g$, $\overline{f} = (x + c)^2 g$, and $p^2 \nmid f(-\tilde{c})$ for some (any) integer lift $\tilde{c}$ of $c$.
\end{itemize}
\end{lemma}

Suppose first $p = 2$. In this case, the only element in $U_n(\F_2)$ is 
\[ x^n + x^{n - 1} + \ldots + 1 = \frac{x^{n + 1} - 1}{x - 1}, \]
    which is squarefree if and only if either $n$ is even or $n = 1$.
Thus $P_{n, 2}^{\sqf}$ equals $0$ if $n \geq 2$ is odd and $1$ if $n$ is even.

Suppose from now on that $p > 2$. Fix some $h\in U_n(\F_p)$. If $h$ is squarefree, then $\Delta(f)$ is a unit for every $f\in U_n(\Z_p)$ with $\bar{f} = h$. The number of such $h\in U_n(\F_p)$ is $$\sum_{u\in\F_p[x]_m}\mu(u)\#\{h\in U_n(\F_p)\colon u^2\mid h\}.$$ 
If $h = (x + c)^2g$ for some nonzero $c$ and squarefree $g\in \F_p[x]_m$ such that $x + c\nmid g$, then among all lifts $f\in U_n(\Z_p)$ with $\bar{f} = h$, a fraction of exactly $1/p$ of them satisfies $p^2\mid f(-\tilde{c})$. The number of such $h\in U_n(\F_p)$ is given by 
\begin{eqnarray*}
    &&\sum_{u\in\F_p[x]_m}\mu(u)\#\{h\in U_n(\F_p)\colon (x + c)^3\nmid h, u^2(x + c)^2\mid h\}\\
    &=&\sum_{\substack{u\in\F_p[x]_m\\ x + c\nmid u}}\mu(u)\Big(\#\{h\in U_n(\F_p)\colon u^2(x + c)^2\mid h\} - \#\{h\in U_n(\F_p)\colon u^2(x + c)^3\mid h\}\Big)
\end{eqnarray*}  Note if $h = x^2g$, then it has no lifts to $U_n(\Z_p)$. Hence, combining the above we have
\begin{eqnarray}
    \nonumber P_{n,p}^{\sqf}
    &=&\sum_{u\in\F_p[x]_m}\mu(u)\frac{\#\{h\in U_n(\F_p)\colon u^2\mid h\}}{|U_n(\F_p)|}\\
    \label{eq:Psqfnp}&&+\left(1 - \frac{1}{p}\right) \sum_{c \in \F_p^{\times}} \sum_{\substack{u \in \F_p[x]_m \\ x + c \nmid u}} \mu(u) \cdot \frac{\# \{h \in U_n(\F_p) : (x + c)^2 u^2 \mid h\}}{|U_n(\F_p)|}\\
    \nonumber&&-\left(1 - \frac{1}{p}\right) \sum_{c \in \F_p^{\times}} \sum_{\substack{u \in \F_p[x]_m \\ x + c \nmid u}} \mu(u) \cdot \frac{\# \{h \in U_n(\F_p) : (x + c)^3 u^2 \mid h\}}{|U_n(\F_p)|}
\end{eqnarray}
To simplify the above expression, we use the following lemma:

\begin{lemma}\label{lem:1}
Let $g \in \F_p[x]_m$ with degree $k$ such that $x \nmid g$.
Then for any $n \geq 2$,
\[ \sum_{u \in \F_p[x]_m} \mu(u) \cdot \frac{\# \{h \in U_n(\F_p) : gu^2 \mid h\}}{|U_n(\F_p)|} = \frac{p}{p + 1} \cdot \frac{1}{p^k} + O\left(\sum_{0\leq d\leq (n-k)/2} p^d \delta_{n,p}(2d + k) + \frac{1}{p^{(n + k)/2}}\right), \]
    and for any $c \in \F_p^{\times}$,
\[ \sum_{\substack{u \in \F_p[x]_m \\ x + c \nmid u}} \mu(u) \cdot \frac{\# \{h \in U_n(\F_p) : gu^2 \mid h\}}{|U_n(\F_p)|} = \frac{p^3}{(p - 1) (p + 1)^2} \cdot \frac{1}{p^k} + O\left(\sum_{0\leq d\leq (n-k)/2} p^d \delta_{n,p}(2d + k) + \frac{1}{p^{(n + k)/2}}\right), \]
    where the implied constants are absolute, and the sums on the right are empty sums if $n < k$.
\end{lemma}

\begin{proof}
Given any $g,u$ not divisible by $x$, we note that
$$\#\{h\in U_n(\F_p)\colon gu^2\mid h\} = \#\{(a_1,\ldots,a_n)\in(\F_p\backslash\{0\})^n\colon a_1x^{n-1} + \cdots + a_n\equiv 0\pmod{gu^2}\}.$$
Hence by definition \eqref{eq:delta} of $\delta_{n,p}(d)$, we see that
$$\left|\frac{\#\{h\in U_n(\F_p)\colon gu^2\mid h\}}{|U_n(\F_p)|} - \frac{1}{p^{2\deg(u) + k}}\right|\leq \delta_{n,p}(2\deg(u) + k).$$
Note that $\#\{h\in U_n(\F_p)\colon gu^2\mid h\}$ can be nonzero only when $2\deg(u) + k\leq n$. Summing over all these $u$ gives
$$\sum_{\substack{u \in \F_p[x]_m \\ 2 \deg(u) + k \leq n}}\delta_{n,p}(2\deg(u) + k) = \sum_{0\leq d\leq (n-k)/2} p^d \delta_{n,p}(2d + k).$$
Thus it remains to show that
\begin{eqnarray*}
    \sum_{\substack{u \in \F_p[x]_m \\ x \nmid u \\ 2 \deg(u) + k \leq n}} \frac{\mu(u)}{p^{2 \deg(u) + k}}
    &=& \frac{p}{p + 1} \cdot \frac{1}{p^k} + O\left(\frac{1}{p^{(n + k)/2}}\right), \\
    \sum_{\substack{u \in \F_p[x]_m \\ x \nmid u, x + c \nmid u \\ 2 \deg(u) + k \leq n}} \frac{\mu(u)}{p^{2 \deg(u) + k}}
    &=& \frac{p^3}{(p - 1)(p + 1)^2} \cdot \frac{1}{p^k} + O\left(\frac{1}{p^{(n + k)/2}}\right).
\end{eqnarray*}
Notice that
\[ \left|\sum_{\substack{u \in \F_p[x]_m \\ 2 \deg(u) + k > n}} \frac{\mu(u)}{p^{2 \deg(u) + k}}\right| \leq \sum_{d > (n-k)/2} \frac{p^d}{p^{2d + k}}\leq \frac{2}{p^{(n + k)/2}}. \]
Thus, it now remains to show that
\begin{equation}\label{eq:LL}
    \sum_{\substack{u \in \F_p[x]_m \\ x \nmid u}} \frac{\mu(u)}{p^{2 \deg(u)}} = \frac{p}{p + 1} \qquad \text{and} \qquad
   \sum_{\substack{u \in \F_p[x]_m \\ x \nmid u, x + c \nmid u}} \frac{\mu(u)}{p^{2 \deg(u)}} = \frac{p^3}{(p - 1)(p + 1)^2}. 
\end{equation} 
   
For any arithmetic function $f : \F_p[x]_m \to \C$, denote its generating series by
\[ L_f(T) = \sum_{u \in \F_p[x]_m} f(u) \; T^{\deg(u)}.  \]
For any $u, v \in \F_p[x]_m$ define $\mathbf{1}_v(u)$ to be $1$ if $v \nmid u$ and $0$ otherwise. The two sums in \eqref{eq:LL} are then given by 
\[ L_{\mathbf{1}_x \cdot \mu}(1/p^2) \qquad \text{and} \qquad L_{\mathbf{1}_x \cdot \mathbf{1}_{x + c} \cdot \mu}(1/p^2). \]
Finally, note that
\[ L_{\mathbf{1}_x \cdot\mu}(T) = \frac{L_{\mu}(T)}{1 + \mu(x) \; T} = \frac{1 - pT}{1 - T},  \]
    and for $c \neq 0$,
\[ L_{\mathbf{1}_x\cdot \mathbf{1}_{x + c} \cdot\mu}(T) = \frac{L_{\mu}(T)}{(1 + \mu(x) \; T)(1 + \mu(x + c) \; T)} = \frac{1 - pT}{(1 - T)^2}.  \]
We are done after setting $T = 1/p^2$.
\end{proof}

Applying Lemma \ref{lem:1} to \eqref{eq:Psqfnp} gives
\begin{align*}
    P_{n,p}^{\sqf}
    &= \frac{p}{p + 1} + \left(1 - \frac{1}{p}\right)(p - 1) \cdot \frac{p^3}{(p - 1)(p + 1)^2} \cdot \frac{1}{p^2} - \left(1 - \frac{1}{p}\right)(p - 1) \cdot \frac{p^3}{(p - 1)(p + 1)^2} \cdot \frac{1}{p^3} \\
    &\qquad+ O\left(\sum_{d \leq  n/2} p^d \delta_{n,p}(2d) + p \sum_{d \leq  (n-2)/2} p^d \delta_{n,p}(2d + 2) + p \sum_{d \leq  (n-3)/2} p^d \delta_{n,p}(2d + 3) + \frac{1}{p^{n/2}} + \frac{p}{p^{(n + 2)/2}} + \frac{p}{p^{(n + 3)/2}}\right) \\
    &= \frac{p}{p + 1} + \left(1 - \frac{1}{p}\right)\frac{p - 1}{(p + 1)^2} + O\left(\sum_{d\leq n/2} p^d \delta_{n,p}(2d) + \sum_{1\leq d\leq (n-1)/2} p^d \delta_{n,p}(2d + 1) + \frac{1}{p^{n/2}}\right) \\
    &= 1 - \frac{3p - 1}{p(p + 1)^2} + O\left(\sum_{d = 0}^n p^{d/2} \delta_{n,p}(d) + \frac{1}{p^{n/2}}\right).
\end{align*}
This proves Theorem~\ref{thm:Psqf}.

\subsection{Proof of Theorem~\ref{thm:Pmax}}

Fix a prime $p$.
We start by stating Dedekind's criterion, which tells us for each $f \in \Z_p[x]$ monic whether $\Z_p[x]/(f(x))$ is the maximal order in $\Q_p[x]/(f(x))$.

\begin{lemma}[Dedekind's criterion](\cite[Corollary 3.2]{ABZ})\label{lem4:maximality-criterion}
Fix a prime $p$.
Let $f$ be a monic polynomial over $\Z_p$.
Then the ring $\Z_p[x]/(f(x))$ is the maximal order in $\Q_p[x]/(f(x))$ if and only if $f \notin (p, g)^2$ for any $g \in \Z_p[x]$ monic, where $(p, g)$ is the ideal of $\Z_p[x]$ generated by $p$ and $g$.
\end{lemma}

\begin{lemma}(\cite[Lemma 3.3]{ABZ})\label{lem4:maximality-inclusion-exclusion-lemma}
Fix $g, h \in \Z_p[x]$ monic and coprime modulo $p$.
Then \[ (p, g)^2 \cap (p, h)^2 = (p, gh)^2. \]
\end{lemma}

Lemma \ref{lem4:maximality-inclusion-exclusion-lemma} implies that we may use the inclusion-exclusion sieve to obtain
\begin{equation}\label{eq:Pmax}
    P_{n,p}^{\max} = \sum_{u\in\F_p[x]_m}\mu(u)\frac{\lambda_p(U_n(\Z_p)\cap(p, \tilde{u})^2)}{\lambda_p(U_n(\Z_p))}.
\end{equation}
where $\tilde{u}$ denotes any monic lift of $u$. If $x\mid u$ or $2\deg(u) > n$, we see that $U_n(\Z_p)\cap(p, \tilde{u})^2 = \emptyset$.

\begin{lemma}\label{lem4:maximality-ideal-measure}
Fix some $u \in \F_p[x]_m$ and a monic lift $\tilde{u} \in \Z_p[x]$ of $u$.
Then $$\frac{\lambda_p(U_n(\Z_p)\cap(p, \tilde{u})^2)}{\lambda_p(U_n(\Z_p))} = p^{-\deg(u)}\frac{\#\{h\in U_n(\F_p)\colon u^2\mid h\}}{|U_n(\F_p)|}.$$ 
\end{lemma}

\begin{proof}
If $x\mid u$ or $2\deg(u) > n$, then both sides are $0$. We may assume $x\nmid u$ and $2\deg(u)\leq n$.
For any $f\in (p,\tilde{u})^2$, we see that $u^2\mid \bar{f}$ in $\F_p[x]$. Hence we fix some $h\in (u^2)$ in $U_n(\F_p)$. Let $f_0$ be a lift of $h$ such that $\tilde{u}^2\mid f_0$, which implies $f_0\in U_n(\Z_p)\cap (p,\tilde{u})^2$. Any other lift of $h$ to $U_n(\Z_p)$ is of the form $f_0 + g$ where $\deg(g)\leq n$ and $g\in (p,\tilde{u})^2\cap (p) = (p)(p,\tilde{u})$. Since $\deg(u)\leq n$, we see that a fraction of $p^{-\deg(u)}$ of the lifts of $h$ to $U_n(\Z_p)$ belongs to $(p,\tilde{u})^2$.
\end{proof}

Applying Lemma \ref{lem4:maximality-ideal-measure} to \eqref{eq:Pmax} gives
\begin{equation}\label{eq:Pmax2}
     P_{n,p}^{\max} = \sum_{u \in \F_p[x]_m} \mu(u) p^{-\deg(u)} \cdot \frac{\#\{h \in U_n(\F_p) : u^2 \mid h\}}{|U_n(\F_p)|}.
\end{equation}
Note that if $x\mid u$ or $2\deg(u) > n$, then the term $\#\{h \in U_n(\F_p) : u^2 \mid h\} = 0$.

We first compute $P_{n,2}^{\max}$ using the above formula. In this case, $U_n(\F_2)$ is the singleton $\{h_0 = x^n + \cdots + x + 1\}$. Then \eqref{eq:Pmax2} simplifies to
\[ P_{n,2}^{\max} = \sum_{\substack{u \in \F_2[x]_m \\ u^2 \mid h_0}} \mu(u) \; 2^{-\deg(u)} = \prod_{\substack{g \in \F_2[x]_m \text{ irreducible} \\ g^2 \mid h_0}} \left(1 - \frac{1}{2^{\deg(g)}}\right). \]
If $n+1$ is odd, then $x^{n + 1} - 1 = (x - 1)h_0$ is squarefree in $\F_2[x]$, in which case the above product is empty, and thus $P_{n,2}^{\max} = 1$.
Suppose now $n+1$ is even. Write $n+1 = 2^{e + 1} t $ for some $e \geq 0$ and $t$ odd.
Then $x^{n + 1} - 1 = (x^t - 1)^{2^{e + 1}}$.
Thus, the irreducible polynomials that are counted in the above product formula are precisely ones dividing $x^t - 1$, except $x - 1$ if $e = 0$.
It remains to show the following:

\begin{lemma}
For any odd positive integer $t$,
\[ \prod_{\substack{g \in \F_2[x]_m \text{ irreducible} \\ g \mid x^t - 1}} \left(1 - \frac{1}{2^{\deg(g)}}\right) = \prod_{d \mid t} \left(1 - \frac{1}{2^{o_d(2)}}\right)^{\phi(d)/o_d(2)}. \]
In fact, both sides are equal to $|R^\times|/|R|$, where $R$ is the ring $\F_2[x]/(x^t - 1)$.
\end{lemma}
\begin{proof}
By the theory of cyclotomic polynomials, we have $x^t - 1 = \prod_{d \mid t} \Phi_d(x)$ and each $\Phi_d(x)$ factors into $\phi(d)/o_d(2)$ monic irreducible polynomials of degree $o_d(2)$ in $\F_2[x]$. All the irreducible factors are distinct because $x^t - 1$ is squarefree.
\end{proof}

We now work on the case where $p$ is odd similar to the calculation of $P^{\sqf}_{n,p}$ in \S3.1. We have the bound
$$\left|\mu(u)p^{-\deg(u)}\left(\frac{\#\{h \in U_n(\F_p) : u^2 \mid h\}}{|U_n(\F_p)|} - \frac{1}{p^{2 \deg(u)}}\right)\right| \leq p^{-\deg(u)}\delta_{n,p}(2\deg(u))$$
for $x\nmid u$. Summing over $u\in\F_p[x]_m$ with $2\deg(u)\leq n$ gives a bound of
$$\sum_{\substack{u\in\F_p[x]_m\\ 2\deg(u)\leq n}}p^{-\deg(u)}\delta_{n,p}(2\deg(u)) \leq \sum_{0\leq d\leq n/2}\delta_{n,p}(2d).$$
Hence, we have
\[ P_{n,p}^{\max} 
    = \sum_{\substack{u \in \F_p[x]_m \\ x \nmid u \\ 2 \deg(u) \leq n}} \frac{\mu(u)}{p^{3 \deg(u)}} + O\left(\sum_{d = 0}^{\lfloor n/2 \rfloor} \delta_{n,p}(2d)\right). \]
Notice that
\[ \left|\sum_{\substack{u \in \F_p[x]_m\\ 2 \deg(u) > n}} \frac{\mu(u)}{p^{3 \deg(u)}}\right| \leq \sum_{d>n/2} \frac{p^d}{p^{3d}} \leq \frac{2}{p^{n}}. \]
Thus, it remains to compute
\[ \sum_{\substack{u \in \F_p[x]_m \\ x \nmid u}} \frac{\mu(u)}{p^{3 \deg(u)}} = L_{\mathbf{1}_x \cdot \mu}(1/p^3) = \frac{1 - 1/p^2}{1 - 1/p^3} = 1 - \frac{1}{p^2 + p + 1}. \]
This proves Theorem~\ref{thm:Pmax}.

\section{Equidistribution of polynomials with nonzero coefficients}

Recall the set
$$\cA_n(u;\alpha) = \{(a_1, a_2, \ldots, a_n) \in (\F_p \setminus \{0\})^n : a_1 x^{n - 1} + \ldots + a_n \equiv \alpha \bmod{u}\}$$
and the discrepancy of equidistribution
$$\delta_{n,p}(d) = \max_{\substack{u \in \F_p[x]_m \\ \deg(u) = d \\ x \nmid u}} \max_{\alpha \in \F_p[x]/(u)} \left|\frac{|\cA_{n}(u; \alpha)|}{(p - 1)^n} - \frac{1}{p^d}\right|.$$
Note that if $(a_1, \ldots, a_n), (b_1, \ldots, b_n) \in \cA_{n}(u; \alpha)$ where $\deg(u) = d$, then \[ a_1 x^{n - 1} + \ldots + a_n \equiv b_1 x^{n - 1} + \ldots + b_n \equiv \alpha \pmod{u}, \]
    and thus \[ u \mid (a_1 - b_1) x^{n - 1} + \ldots + (a_n - b_n). \]
By degree constraints, if $a_i = b_i$ for all $i \leq n - d$, then $(a_1, \ldots, a_n) = (b_1, \ldots, b_n)$. Hence the projection map $\cA_{n}(u; \alpha) \to (\F_p \setminus \{0\})^{n - d}$ defined by projecting into the first $n - d$ components is injective.
This proves that $|\cA_{n}(u; \alpha)| \leq (p - 1)^{n - d}$ and so we have the ``trivial'' bound
\begin{equation}
    \label{eq:triv}
    \delta_{n,p}(d) \leq \frac{1}{(p-1)^d}.
\end{equation}
The goal of this section is to prove the following estimates.

\begin{theorem}\label{thm3:matrix-p-large-bound}
For any odd prime $p$ and integers $n, d \geq 1$ with $n \geq 2d$, we have
\[ \delta_{n,p}(d) \leq \frac{1}{p^d} \left(\left(\frac{p}{p - 1}\right)^d - 1\right)^{\lfloor n/(2d) \rfloor}. \]
In particular, if $d\leq 4$ and $n \geq 16$, then
\[ \delta_{n,p}(d) \ll \frac{1}{p^{d + 2}} \]
with the implied constant absolute.
\end{theorem}

\begin{theorem}\label{thm3:matrix-p-small-bound}
For any odd prime $p$ and integers $n, d \geq 1$, we have
\[ \delta_{n,p}(d) \leq e^{1/3} \exp\left(-\frac{n}{(d^2 + d) p^{d^2}}\right). \]
\end{theorem}

\subsection{Doubly stochastic matrices}

Fix some $u \in \F_p[x]_m$ such that $x \nmid u$. Let $G_u$ be the directed graph with vertex set $V(G_u) = \F_p[x]/(u)$ and an edge from $\alpha$ to $\beta$ if $\beta = \alpha x + c$ for some $c \in \F_p \setminus \{0\}$ (where we allow loops).
Notice that a path of length $n$ from a vertex $\alpha$ to another vertex $\beta$ corresponds to $(a_1, a_2, \ldots, a_n) \in (\F_p \setminus \{0\})^n$ such that
\[ \beta \equiv \alpha x^n + a_1 x^{n - 1} + a_2 x^{n - 2} + \ldots + a_n \pmod{u}. \]
The graph $G_u$ is $(p - 1)$-regular.
That is, each vertex has out-degree $p - 1$, and since $x$ is invertible mod $u$, each vertex also has in-degree $p - 1$.

Let $M_u$ denote the adjacency matrix of $G_u$, and let $\widetilde{M}_u = (p - 1)^{-1} M_u$.
The $(\beta, \alpha)$-entry of $M_u$ is $1$ if $\beta = \alpha x + c$ for some $c \in \F_p$ non-zero, and $0$ otherwise.
Since the out-degree and in-degree of every vertex is $p-1$, we see that $\widetilde{M}_u$ is a doubly stochastic matrix.
That is, its entries are all non-negative, and the sum of the entries in each row or column is $1$.
The $(\beta,\alpha)$-entry of $M_u^n$ is the number of paths from $\alpha$ to $\beta$. In particular,
\[ \mbox{the }(\beta, 0)\mbox{-entry of }M_u^n = \#\{(a_1, \ldots, a_n) \in (\F_p \setminus \{0\})^n : a_1 x^{n - 1} + \ldots + a_n \equiv \beta \bmod{u}\} = |\cA_{n}(u; \beta)|. \]
Let $J_N$ be the all-one matrix of dimension $N = |V(G_u)| = p^{\deg(u)}$, and let $\tilde{J}_N = (1/N) J_N$.
Then the $(\beta,0)$-entry of $\widetilde{M}_u^n - \tilde{J}_N$ is exactly
$$\frac{|\cA_n(u;\beta)|}{(p-1)^n} - \frac{1}{p^{\deg(u)}}.$$
Therefore, $\delta_{n,p}(d)$ is bounded above by the maximum of the absolute values of the entries of $\widetilde{M}_u^n - \tilde{J}_N$ as $u$ varies in $\F_p[x]_m$ such that $\deg(u) = d$ and $x\nmid u$.

We start by recalling some general theory on doubly stochastic matrix.

\begin{definition}
    {\em A real matrix $M$ is a \emph{doubly stochastic matrix} if its entries are non-negative and the sum of the entries in each row or column is $1$.}
\end{definition}

Suppose $M\in M_m(\R)$. Let $J_m$ be the all-one matrix of dimension $m$ and let $\tilde{J}_m = (1/m) J_m$. Then $\tilde{J}_m^2 = \tilde{J}_m$. Moreover, if all the entries of $M$ are non-negative, then $M$ is doubly stochastic if and only if $M \tilde{J}_m = \tilde{J}_m M = \tilde{J}_m$. Given two doubly stochastic matrices $M_1$ and $M_2$ of dimension $m$, $M_1 M_2$ is also a doubly stochastic matrix, and we have the formula
\[ (M_1 - \tilde{J}_m)(M_2 - \tilde{J}_m) = M_1 M_2 - \tilde{J}_m. \]
In particular, by induction, if $M$ is a doubly stochastic matrix of dimension $m$, then $M^k - \tilde{J}_m = (M - \tilde{J}_m)^k$ for any positive integer $k$.

For any matrices $M \in M_m(\R)$, we denote the minimum entry, maximum entry, and the normalized max norm respectively by
\[ E_{\min}(M) = \min_{i, j} M_{ij}, \qquad E_{\max}(M) = \max_{i, j} M_{ij}, \qquad \|M\|_{\max} = m\cdot\max_{i, j} |M_{ij}|, \]
    where $M_{ij}$ is the $(i, j)$-entry of $M$.
It is easy to see that
\begin{equation}\label{eq:Mmax}
    \|M\|_{\max} = m \cdot \max\{|E_{\min}(M)|, |E_{\max}(M)|\}. 
\end{equation} 
We prove first that multiplying by a doubly stochastic matrix does not increase the normalized max norm.

\begin{lemma}\label{lem3:DSmatrix-2}
For any $M_1, M_2 \in M_m(\R)$ with $M_2$ doubly stochastic, we have
\[ E_{\min}(M_1 M_2), E_{\min}(M_2 M_1) \geq E_{\min}(M_1) \quad \text{ and } \quad E_{\max}(M_1 M_2), E_{\max}(M_2 M_1) \leq E_{\max}(M_1).
\]
In particular, $\|M_1 M_2\|_{\max} \leq \|M_1\|_{\max}$ and $\|M_2 M_1\|_{\max} \leq \|M_1\|_{\max}$.
\end{lemma}
\begin{proof}
We first prove $E_{\min}(M_1 M_2) \geq E_{\min}(M_1)$.
Indeed, for any $i, j$, since $(M_2)_{ij}\geq 0$,
\[ (M_1 M_2)_{ij} = \sum_{k = 1}^m (M_1)_{ik} (M_2)_{kj} \geq E_{\min}(M_1) \sum_{k = 1}^m (M_2)_{kj} = E_{\min}(M_1). \]
By similar method, we get $E_{\min}(M_2 M_1) \geq E_{\min}(M_1)$ and $E_{\max}(M_1 M_2), E_{\max}(M_2 M_1) \leq E_{\max}(M_1)$. Then from $$E_{\min}(M_1)\leq E_{\min}(M_1M_2)\leq E_{\max}(M_1M_2)\leq E_{\max}(M_1)$$
and \eqref{eq:Mmax}, we get $\|M_1 M_2\|_{\max}\leq \|M_1 \|_{\max}$ and similarly for $\|M_2 M_1\|_{\max}$.
\end{proof}

Next we prove that the normalized max norm is sub-multiplicative.
\begin{lemma}\label{lem3:DSmatrix-3}
For any $M_1, M_2 \in M_m(\R)$, we have $\|M_1 M_2\|_{\max} \leq \|M_1\|_{\max} \|M_2\|_{\max}$.
In particular, for any $M \in M_m(\R)$ and $n \geq 0$, we have $\|M^n\|_{\max} \leq \|M\|_{\max}^n$.
\end{lemma}
\begin{proof}
The second statement follows from the first statement by induction on $n$.
Thus, it remains to prove the first statement.
For any $i, j \leq m$, we have
\[ |(M_1 M_2)_{ij}| \leq \sum_{k = 1}^m |(M_1)_{ik}| |(M_2)_{kj}| \leq m \cdot \frac{\|M_1\|_{\max}}{m} \cdot \frac{\|M_2\|_{\max}}{m} = \frac{\|M_1\|_{\max} \|M_2\|_{\max}}{m}, \]
    and thus $\|M_1 M_2\|_{\max} \leq \|M_1\|_{\max} \|M_2\|_{\max}$.
\end{proof}

\begin{corollary}\label{cor3:DSmatrix-4}
For any $M \in M_m(\R)$ doubly stochastic and integers $k \geq 0$ and $\ell > 0$, we have
\[ \|M^k - \tilde{J}_m\|_{\max} \leq \|M^\ell - \tilde{J}_m\|_{\max}^{\lfloor k/\ell \rfloor}. \]
\end{corollary}
\begin{proof}
We can write $k = \ell q + r$, where $q = \lfloor k/\ell \rfloor$ and $\ell > r \geq 0$.
Then $M^r$ is doubly stochastic and we have $$\|M^k - \tilde{J}_m\|_{\max} = \|M^r(M^{\ell q} - \tilde{J}_m)\|_{\max}\leq \|(M^\ell - \tilde{J}_m)^q\|_{\max}\leq \|M^\ell - \tilde{J}_m\|^q_{\max}$$ by Lemma~\ref{lem3:DSmatrix-2} and Lemma~\ref{lem3:DSmatrix-3}.
\end{proof}

We conclude with a lower bound on $E_{\min}(M^n)$ in terms of $E_{\min}(M)$.

\begin{lemma}\label{lem3:Emin-product-stochastic-bound}
For any doubly stochastic matrices $M_1, M_2 \in M_m(\R)$,
\[ 1 - m E_{\min}(M_1 M_2) \leq (1 - m E_{\min}(M_1))(1 - m E_{\min}(M_2)). \]
In particular, for any doubly stochastic matrix $M \in M_m(\R)$ and $n \geq 0$,
\[ E_{\min}(M^n) \geq \frac{1 - (1 - m E_{\min}(M))^n}{m}. \]
\end{lemma}
\begin{proof}
Assume $m > 1$, as we are done otherwise.
The second statement follows from the first statement by induction on $n$.
We now show the first statement.

We write $r_1 = E_{\min}(M_1)$ and $r_2 = E_{\min}(M_2)$.
Suppose $(M_1 M_2)_{\min}=\sum_{i = 1}^m a_i b_i$, where the $a_i$'s are entries of a row of $M_1$ and the $b_i$'s are entries of a column of $M_2$.
In particular, $a_i \geq r_1$ and $b_i \geq r_2$ for each $i \leq m$.
By the rearrangement inequality, it suffices to prove that
$$1 - m\sum_{i=1}^ma_ib_i \leq (1 - mr_1)(1 - mr_2)$$
under the assumption that $a_1 \geq a_2 \geq \ldots \geq a_m \geq r_1$ and $r_2 \leq b_1 \leq b_2 \leq \ldots \leq b_m$.
We have
\begin{align*}
    \sum_{i = 1}^m a_i b_i
    &= \left(1 - \sum_{i = 2}^m a_i\right) b_1 + \sum_{i = 2}^{m - 1} a_i b_i + a_m \left(1 - \sum_{i = 1}^{m - 1} b_i\right) \\
    &= b_1 + a_m - 2 b_1 a_m + \sum_{i = 2}^{m - 1} (a_i b_i - a_i b_1 - a_m b_i) \\
    &= b_1 + a_m - 2 b_1 a_m + \sum_{i = 2}^{m - 1} ((a_i - a_m)(b_i - b_1) - b_1 a_m) \\
    &\geq b_1 + a_m - m b_1 a_m.
\end{align*}
Then
$$ 1 - m \sum_{i = 1}^m a_i b_i \leq 1 - mb_1 - ma_m + m^2 b_1 a_m = (1 - mb_1)(1 - ma_m)\leq(1 - mr_1)(1 - mr_2), $$
where the last inequality follows because $$1 - mr_2\geq1 - mb_1\geq 1-b_1-\cdots-b_m = 0$$ and similarly $1 - mr_1\geq 1 - ma_m\geq 0$.
\end{proof}

\subsection{Proof of Theorem~\ref{thm3:matrix-p-large-bound} and Theorem~\ref{thm3:matrix-p-small-bound} }

We now prove Theorem~\ref{thm3:matrix-p-large-bound}.
We start by proving:

\begin{lemma}\label{thm3:Mu-Emin-bound-p-large}
For any $u \in \F_p[x]_m$ of degree $d$ such that $x \nmid u$, we have
\[ E_{\min}(M_u^{2d}) \geq 2 (p - 1)^d - p^d \quad \text{ and } \quad E_{\max}(M_u^{2d}) \leq (p - 1)^d. \]
In particular, 
\[ \|\widetilde{M}_u^{2d} - \tilde{J}_{p^d}\|_{\max} \leq \left(\frac{p}{p - 1}\right)^d - 1. \]
\end{lemma}
\begin{proof}
We write
\[ S_0 = \{a_{d - 1} x^{d - 1} + \ldots + a_0 \bmod{u} : a_0, a_1, \ldots, a_{d - 1} \in \F_p \setminus \{0\}\} \subseteq V(G_u). \]
Fix some $\alpha, \beta \in V(G_u)$.
Then the $(\beta, \alpha)$-entry $(M_u^{2d})_{\beta, \alpha}$ of $M_u^{2d}$ is the number of $2d$-tuples $(b_0, b_1, \ldots, b_{2d - 1}) \in (\F_p \setminus \{0\})^{2d}$ such that
\[ \alpha x^{2d} + b_{2d - 1} x^{2d - 1} + \ldots + b_0 \equiv \beta\pmod{u}, \]
    which is equivalent to
\[ \sum_{i = 0}^{d - 1} b_i x^i \equiv \beta - \alpha x^{2d} - \left(\sum_{i = 0}^{d - 1} b_{i + d} x^i\right) x^d\pmod{u}. \]
That is, we have
\[ (M_u^{2d})_{\beta, \alpha} = \#\{\gamma \in S_0 : \beta - \alpha x^{2d} - \gamma x^d \in S_0\}. \]

Consider the map $f : V(G_u) \to V(G_u)$ given by $\gamma \mapsto \beta - \alpha x^{2d} - \gamma x^d$.
Then the above equality simplifies to
\[ (M_u^{2d})_{\beta, \alpha} = \#\{\gamma \in S_0 : f(\gamma) \in S_0\} = |f(S_0) \cap S_0|. \]
On one hand, we have
\[ |f(S_0) \cap S_0| \leq |S_0| = (p - 1)^d. \]
On the other hand, $f$ is injective since $x$ is invertible mod $u$, so
\[ |f(S_0) \cap S_0| = |f(S_0)| + |S_0| - |f(S_0) \cup S_0| \geq 2 (p - 1)^d - p^d. \]
Thus $2 (p - 1)^d - p^d \leq (M_u^{2d})_{\beta, \alpha} \leq (p - 1)^d$, as desired. Dividing by $(p-1)^{2d}$ gives
\[ E_{\min}(\widetilde{M}_u^{2d}) \geq \max\left\{0,\frac{2}{(p - 1)^d} - \frac{p^d}{(p - 1)^{2d}}\right\} \quad \text{ and } \quad E_{\max}(\widetilde{M}_u^{2d}) \leq \frac{1}{(p - 1)^d}. \]
This means that
\begin{align*}
    E_{\min}(\widetilde{M}_u^{2d}) - \frac{1}{p^d}
        &\geq \max\left\{-\frac{1}{p^d},\frac{2}{(p - 1)^d} - \frac{p^d}{(p - 1)^{2d}} - \frac{1}{p^d}\right\}
        = -\frac{1}{p^d} \min\left\{1,\left(\left(\frac{p}{p - 1}\right)^d - 1\right)^2\right\} \\
    E_{\max}(\widetilde{M}_u^{2d}) - \frac{1}{p^d}
        &\leq \frac{1}{(p - 1)^d} - \frac{1}{p^d}
        = \frac{1}{p^d} \left(\left(\frac{p}{p - 1}\right)^d - 1\right).
\end{align*}
As a result, by definition of the normalized max norm \eqref{eq:Mmax},
\[ \|\widetilde{M}_u^{2d} - \tilde{J}_{p^d}\|_{\max} \leq \left(\frac{p}{p - 1}\right)^d - 1, \]
    as desired.
\end{proof}

\begin{proof}[Proof of Theorem~\ref{thm3:matrix-p-large-bound}] By Corollary \ref{cor3:DSmatrix-4}, we know that $$\|\widetilde{M}_u^n - \tilde{J}_{p^d}\|_{\max} \leq \|\widetilde{M}_u^{2d} - \tilde{J}_{p^d}\|^{\lfloor n/(2d)\rfloor}_{\max}\leq \left(\left(\frac{p}{p-1}\right)^d - 1\right)^{\lfloor n/(2d)\rfloor}.$$
Dividing by $p^d$ gives the bound 
\[ \delta_{n,p}(d) \leq \frac{1}{p^d} \left(\left(\frac{p}{p - 1}\right)^d - 1\right)^{\lfloor n/(2d) \rfloor}. \]
Suppose now $n\geq 16$ and $d\leq 4$. Then similar to the proof of Corollary \ref{cor3:DSmatrix-4}, we have
$$\|\widetilde{M}_u^n - \tilde{J}_{p^d}\|_{\max} = \|\widetilde{M}_u^{n-4d}(\widetilde{M}_u^{4d} - \tilde{J}_{p^d})\|_{\max}\leq\|(\widetilde{M}_u^{2d} - \tilde{J}_{p^d})^2\|_{\max}\leq \left(\left(\frac{p}{p-1}\right)^d - 1\right)^2 \ll\frac{1}{p^2}.$$
Dividing by $p^d$ gives the desired result.
\end{proof}

We now prove Theorem~\ref{thm3:matrix-p-small-bound}.
This time, we need:

\begin{lemma}\label{lem3:Mu-Emin-bound-p-small}
For any $u \in \F_p[x]_m$ of degree $d$ such that $x \nmid u$, the entries of $M_u^{d^2 + d}$ are all positive.
\end{lemma}
\begin{proof}
Let $S_0$ be as in the proof of Theorem~\ref{thm3:Mu-Emin-bound-p-large}.
Fix some $\alpha, \beta \in V(G_u)$.
The $(\beta, \alpha)$-entry $(M_u^{d^2 + d})_{\beta, \alpha}$ of $M_u^{d^2 + d}$ is the number of $(d^2 + d)$-tuples $(b_0, b_1, \ldots, b_{d^2 + d - 1}) \in (\F_p \setminus \{0\})^{d^2 + d}$ such that
\[ \alpha x^{d^2 + d} + b_{d^2 + d - 1} x^{d^2 + d - 1} + \ldots + b_0 \equiv \beta\pmod{u}, \]
    which is equivalent to
\[ \sum_{i = 0}^{d - 1} b_i x^i \equiv \beta - \alpha x^{d^2 + d} - \sum_{j = 1}^d \left(\sum_{i = 0}^{d - 1} b_{i + jd} x^i\right) x^{jd}\pmod{u}. \]
Thus, it suffices to show that there exists $\gamma_1, \ldots, \gamma_d \in S_0$ such that
\[ \beta - \alpha x^{d^2 + d} - \sum_{j = 1}^d \gamma_j x^{jd} \in S_0. \]

Since $x\nmid u$, we know that $x^{-1}\in\F_p[x]/(u)$. First we find $\eta_1, \eta_2, \ldots, \eta_d \in S_0$ such that $\eta_j + x^{j(1 - d) - 1} \in S_0$ for each $j \leq d$.
Indeed, fix $j$, and let $c_{j, 0}, c_{j, 1}, \ldots, c_{j, d - 1} \in \F_p$ such that
\[ x^{j(1 - d) - 1} \equiv \sum_{i = 0}^{d - 1} c_{j, i} x^i \pmod{u}. \]
Since $p>2$, there exist $a_{j, i} \in \F_p \setminus \{0, -c_{j, i}\}$ for each $i < d$. Then we have
\[ \sum_{i = 0}^{d - 1} a_{j, i} x^i \in S_0 \quad \text{ and } \quad \sum_{i = 0}^{d - 1} (a_{j, i} + c_{j, i}) x^i \in S_0, \]
We take $\eta_j = a_{j, d - 1} x^{d - 1} + \ldots + a_{j, 1} x + a_{j, 0}$ so that $\eta_j + c_j x^{j(1 - d) - 1} \in S_0$ for both $c_j=0$ and $c_j =1$.
Write \[ \beta - \alpha x^{d^2 + d} - \sum_{j = 1}^d \eta_j x^{jd} \equiv a_{d - 1} x^{d - 1} + \dots + a_1 x + a_0 \pmod{u}\]
    for some $a_{d - 1}, \ldots, a_0 \in \F_p$.
For any $j = 1,\ldots,d$, we choose $c_j = 1$ if $a_{j-1} = 0$; and $c_j = 0$ otherwise. Then we see that \[ \beta - \alpha x^{d^2 + d} - \sum_{j = 1}^d (\eta_j + c_j x^{j(1 - d) - 1}) x^{jd} \equiv \sum_{j = 1}^d (a_{j-1} - c_j) x^{j - 1} \in S_0. \]
The lemma is proved by taking $\gamma_j = \eta_j + c_j x^{j(1 - d) - 1} \in S_0$.
\end{proof}

\begin{proof}[Proof of Theorem~\ref{thm3:matrix-p-small-bound}]
Consider an arbitrary $u \in \F_p[x]_m$ of degree $d$ such that $x \nmid u$.
By Lemma~\ref{lem3:Mu-Emin-bound-p-small}, we have
\[ E_{\min}\left(\widetilde{M}_u^{d^2 + d}\right) \geq \frac{1}{(p - 1)^{d^2 + d}} \geq \frac{1}{p^{d^2 + d}}. \]
By Lemma~\ref{lem3:Emin-product-stochastic-bound} and Lemma~\ref{lem3:DSmatrix-2},
\[ E_{\min}(\widetilde{M}_u^n) \geq \frac{1}{p^d}
\left(1 - \Big(1 - \frac{p^d}{p^{d^2 + d}}\Big)^{\lfloor n/(d^2 + d) \rfloor}\right) \geq \frac{1}{p^d} - \frac{1}{p^d} \exp\left(-\frac{\lfloor n/(d^2 + d) \rfloor}{p^{d^2}}\right). \]
Since $\widetilde{M}_u^n$ is doubly stochastic, it then follows that
\[ E_{\max}(\widetilde{M}_u^n) \leq \frac{1}{p^d} + \frac{p^d - 1}{p^d}\exp\left(-\frac{\lfloor n/(d^2 + d) \rfloor}{p^{d^2}}\right) \leq \frac{1}{p^d} + \exp\left(-\frac{\lfloor n/(d^2 + d) \rfloor}{p^{d^2}}\right). \]
As a result, the entries of $\widetilde{M}_u^n$ are of distance at most $\displaystyle \exp\left(-\frac{\lfloor n/(d^2 + d) \rfloor}{p^{d^2}}\right)$ from $p^{-d}$.
Since this holds for all $u \in \F_p[x]_m$ of degree $d$ such that $x \nmid u$,
\[ \delta_{n,p}(d) \leq \exp\left(-\frac{\lfloor n/(d^2 + d) \rfloor}{p^{d^2}}\right) \leq e^{p^{-d^2}} \exp\left(-\frac{n}{(d^2 + d) p^{d^2}}\right). \]
We are done, since $d \geq 1$ and $p \geq 3$.
\end{proof}

\subsection{Proof of Theorem~\ref{thm:polydisc} for squarefree density}

We first consider $P^{\sqf}_{n,p}$. By Theorem \ref{thm:Psqf}, it remains to prove that when $n\geq 16$,
\begin{equation}
    \sum_{d = 0}^n p^{d/2} \delta_{n,p}(d) \ll \min\left\{\left(\frac{p}{(p - 1)^2}\right)^{5/2}, \left(\frac{p}{(p - 1)^2}\right)^{\lfloor\sqrt{\log n/\log p}\rfloor/2}\right\}, \label{eq4:local-sqfree-discr-error} 
\end{equation} 
    with the implied constant being absolute.

Using the bound $\delta_{n,p}(d)\ll 1/p^{d+2}$ for $d\leq 4$ from Theorem \ref{thm3:matrix-p-large-bound} and the trivial bound \eqref{eq:triv}, we have
\[ \sum_{d = 0}^n p^{d/2} \delta_{n,p}(d) \ll \sum_{d = 1}^4 \frac{p^{d/2}}{p^{d + 2}} + \sum_{d = 5}^{\infty} \frac{p^{d/2}}{(p - 1)^d} \ll \frac{1}{p^{5/2}} + \left(\frac{p}{(p - 1)^2}\right)^{5/2} \ll \left(\frac{p}{(p - 1)^2}\right)^{5/2}. \]
Suppose now $\log n/\log p \geq 25$, which is when the second term in \eqref{eq4:local-sqfree-discr-error} is smaller. Let $N_0 = \lfloor\sqrt{\log n/\log p}\rfloor$. By Theorem \ref{thm3:matrix-p-small-bound} and the trivial bound \eqref{eq:triv}, we have
\begin{eqnarray*}
    \sum_{d = 0}^n p^{d/2} \delta_{n,p}(d)
    &\ll& \sum_{d = 1}^{N_0 - 1} p^{d/2} \exp\left(-\frac{n}{d(d + 1) p^{d^2}}\right) + \sum_{d = N_0}^{\infty} \frac{p^{d/2}}{(p - 1)^d} \\
    &\ll& p^{N_0/2} \exp\left(-\frac{n}{(N_0 - 1) N_0 p^{(N_0 - 1)^2}} + N_0 \right) + \left(\frac{p}{(p - 1)^2}\right)^{N_0/2}.
\end{eqnarray*}
Choosing $N_0^2\leq\log n/\log p$ guarantees that
$$\frac{\log n}{\log p} \geq (N_0 - 1)^2 + 2(N_0 - 1) + 1 \geq (N_0 - 1)^2 + \frac{3\log N_0}{\log p} + \frac{\log(\log(p-1)+1)}{\log p}$$
which implies that
$$n\geq p^{(N_0 - 1)^2}N_0^3(\log(p-1)+1)$$
and so
$$-\frac{n}{(N_0 - 1)N_0p^{(N_0 - 1)^2}} \leq -N_0\log(p-1) - N_0.$$
Hence, we have $$\sum_{d = 0}^n p^{d/2} \delta_{n,p}(d)
    \ll\left(\frac{p}{(p - 1)^2}\right)^{\lfloor\sqrt{\log n/\log p}\rfloor/2}.$$

For $n\geq 16$, since the product
$$\prod_p \left(1 - \frac{3p-1}{p(p+1)^2} + O\left(\Big(\frac{p}{(p-1)^2}\Big)^{5/2}\right)\right)$$
converges, we conclude from the Lebesgue dominated convergence theorem that
$$\lim_{n\rightarrow\infty}\left(\prod_{p>2}P^{\sqf}_{n,p}\right) = \prod_{p>2}\left(\lim_{n\rightarrow\infty}P^{\sqf}_{n,p}\right) = \prod_{p>2} \left(1 - \frac{3p-1}{p(p+1)^2}\right).$$

For arbitrary $n\geq 2$ and any prime $p>2$, we know that $x^{n+1} - 1$ is squarefree in $\F_p[x]$ if $p\nmid n+1$ and $x^{n+1} + x - 2$ is squarefree in $\F_p[x]$ if $p\mid n+1$. Dividing by $x - 1$, we see that either $x^n + x^{n-1} + \cdots + x + 1$ or $x^n + x^{n-1} + \cdots + x + 2$ has squarefree discriminant and both have all their coefficients not divisible by $p$. Hence by the last statement of Theorem \ref{thm:primesieve}, we know that $C_{n,2}^{\sqf}$ is positive.

\subsection{Proof of Theorem~\ref{thm:polydisc} for maximality density}

We prove the similar statements for $P^{\max}_{n,p}$. By Theorem \ref{thm:Pmax}, it remains to prove that when $n\geq 16$,
\begin{equation}
    \sum_{d = 0}^{\lfloor n/2 \rfloor} \delta_{n,p}(2d) \ll \min\left\{\frac{1}{(p - 1)^4}, \frac{1}{(p - 1)^{2\lfloor\sqrt{\log n/(4\log p)}\rfloor}}\right\}, \label{eq4:local-maximality-error}
\end{equation}
    with the implied constant being absolute.

For the first bound, we use the estimate $\delta_{n,p}(2) \ll 1/p^{4}$ by Theorem~\ref{thm3:matrix-p-large-bound} and the trivial bound $\delta_{n,p}(2d) \leq 1/(p - 1)^{2d}$ to obtain
\[ \sum_{d = 0}^{\lfloor n/2 \rfloor} \delta_{n,p}(2d) \ll \frac{1}{p^4} + \sum_{d = 2}^{\infty} \frac{1}{(p - 1)^{2d}} \ll \frac{1}{(p - 1)^4}. \]

Now we prove the second bound when $\log n\geq 16\log p$. Let $N_0 = \lfloor \sqrt{\log n/(4 \log p)} \rfloor$.
By Theorem~\ref{thm3:matrix-p-small-bound} and the trivial bound \eqref{eq:triv}, we have
\begin{eqnarray*}
    \sum_{d = 0}^{\lfloor n/2 \rfloor} \delta_{n,p}(2d)
    &\ll& \sum_{d = 1}^{N_0 - 1} \exp\left(-\frac{n}{2d(2d + 1) p^{4d^2}}\right) + \sum_{d = N_0}^{\infty} \frac{1}{(p - 1)^{2d}} \\
    &\ll& \exp\left(-\frac{n}{(2 N_0 - 2)(2 N_0 - 1) p^{4(N_0 - 1)^2}}+2N_0\right) + \frac{1}{(p - 1)^{2 N_0}}.
\end{eqnarray*}
Choosing $N_0^2\leq \log n/(4\log p)$ guarantees
$$\frac{\log n}{\log p}\geq 4(N_0 - 1)^2 + \frac{3\log(2N_0)}{\log p} + \frac{\log(\log(p-1)+1)}{\log p}$$
which implies that
$$n\geq p^{4(N_0 - 1)^2}(8N_0^3)(\log(p-1) + 1)$$
and so
$$-\frac{n}{(2N_0-2)(2N_0-1)p^{4(N_0 - 1)^2}} \leq -2N_0(\log(p-1)) - 2N_0.$$
Hence
$$ \sum_{d = 0}^{\lfloor n/2 \rfloor} \delta_{n,p}(2d) \ll \frac{1}{(p - 1)^{2 N_0}}.$$

Finally, the proofs for the convergence and the positivity of $C^{\max}_{n,2}$ are exactly the same as those for $C^{\sqf}_{n,2}$ in \S4.3.

\subsection*{Acknowledgments}
The first named author is supported by the University of Waterloo through an MURA project. The second and third named authors are supported by an NSERC Discovery Grant.

\end{document}